\documentclass[12pt]{article}
\usepackage{fullpage}
\usepackage{amsmath,amssymb}
\usepackage{amsthm}

\newtheorem{cor}{Corollary}
\newtheorem{theo}{Theorem}
\newtheorem{lem}{Lemma}

\title{Choosability with separation of complete multipartite graphs and hypergraphs}

\author{{\bf Zolt\'an F\"uredi}\\
{\small\em
University of Illinois at Urbana--Champaign, Urbana, IL 61801}\\
{\small\em and
Mathematical Institute of the Hungarian Academy of Sciences, Budapest 1364, Hungary}
\thanks{ Research supported in part by the Hungarian National Science Foundation OTKA, 
 by the National Science Foundation under grant NFS DMS 09-01276, 
  and by the European Research Council Advanced Investigators Grant 267195.} \\
{\small\texttt furedi@renyi.hu, z-furedi@math.uiuc.edu}
\\ \\
{\bf Alexandr Kostochka}\\
{\small\em
University of Illinois at Urbana--Champaign, Urbana, IL 61801}\\
{\small\em and Sobolev Institute of Mathematics, Novosibirsk 630090, Russia}
\thanks{ Research of this author
is supported in part by NSF grant DMS-0965587 and by
the Ministry of education and science of the Russian Federation (Contract no. 14.740.11.0868).}\\
{\small\texttt kostochk@math.uiuc.edu}\\
\\
{\bf Mohit Kumbhat}\\
{\small\em University of Illinois at Urbana--Champaign, Urbana, IL 61801}
\thanks{ Research was done partly at the R\'enyi Mathematical Institute, Budapest, Hungary.} 
\\
{\small\texttt kumbhat2@illinois.edu} \\
}

\date{${}$}

\begin{document}

\maketitle
\begin{abstract}
For a hypergraph $G$ and a positive integer $s$, let $\chi_{\ell} (G,s)$ be the minimum value of $l$ such
that $G$ is $L$-colorable from every  list $L$
with $|L(v)|=l$ for each $v\in V(G)$ and $|L(u)\cap L(v)|\leq s$ for all $u, v\in e\in E(G)$.
 This parameter was studied by Kratochv\'{i}l, Tuza and Voigt 
 for various kinds of graphs.
Using randomized constructions 
  we find the asymptotics  of $\chi_{\ell} (G,s)$ for   balanced  complete multipartite graphs 
and for complete $k$-partite $k$-uniform hypergraphs. 
\\
\emph{Key words}: multipartite graphs, hypergraphs, list coloring, separation.\\
\emph{2010 Mathematics Subject Classification}: 05C15.
\end{abstract}

\section{Introduction}

Given a hypergraph $G$, a {\em list } $L$ for $G$ is an assignment to every $v\in V(G)$
 of a set $L(v)$ of colors that may be used for the coloring of $v$.
 We say that $G$ is $L$-colorable, if there exists a proper coloring
 $f$ of the vertices of $G$ from $L$, i.e. if $f(v)\in L(v)$ for all
 $v\in V(G)$ and no edge of $G$ is monochromatic in $f$. A list
 $L$ for a hypergraph $G$ is a $k$-{\em list} if $|L(v)|=k$
  for all $v\in V(G)$.
  An extensively studied parameter is the
  $\textit{list chromatic number}$ of $G$, $\chi_{l}(G)$,
 introduced by Vizing~\cite{V} and Erd\H os, Rubin and Taylor~\cite{ERT}.
For a hypergraph $G$,   $\chi_{l}(G)$
 is
  the least $k$ such that $G$ is $L$-colorable for every $k$-list $L$.
This parameter is also sometimes called
  $\textit{choice number}$, or $\textit{choosability}$ of $G$. \medskip

By definition, $\chi_{\ell} (G)\ge \chi(G)$ for any hypergraph $G$. Moreover,
$\chi_{\ell} (G)$ may be much larger than $\chi(G)$. For example,
$\chi_{\ell} (K_{n,n})$ has the order of $\log n$ (see, e.g.,~\cite{A}),
while, by definition, $\chi(K_{n,n})=2$.
 It is natural to ask what happens when the lists of adjacent vertices in (hyper)graphs
 do not intersect too much. \medskip

For a positive integer  $s$, a list $L$ for a hypergraph $G$ is  $s$-{\em separated} if
 $|L(u)\cap L(v)|\le s$ for all pairs $\{u,v\}$ such that some edge of $G$ contains both, $u$ and $v$.
If $G$ is a graph, this means that for each $uv \in E(G)$, $L(u)$ and $L(v)$ share at most $s$ colors.
 Let $\chi_{\ell}(G,s)$ denote the minimum
$k$ such that $G$ is $L$-colorable from each $s$-separated $k$-list  $L$.
 By definition, for every $1\leq s_1\leq s_2$,
\begin{equation}\label{j1}
\chi_{\ell}(G,s_1)\leq \chi_{\ell}(G,s_2)\leq \chi_{\ell}(G).
\end{equation}

 Kratochv{\'i}l, Tuza and Voigt~\cite{KTV} studied $\chi_{\ell} (G,s)$ for various $G$ and $s$.
  They showed the following.

\begin{theo}[\cite{KTV}]\label{thm1} For positive integers $s,n$ with $s\leq n$,
$\sqrt{\frac{1}{2}sn}\le \chi_{\ell} ({K_n,s})\le \sqrt{2esn}$.
\end{theo}

So the ratio of the upper and lower bounds is $2\sqrt{e}\sim 3.29$. In~\cite{FKK},
 the asymptotics of $\chi_{\ell} ({K_n,s})$
 for every fixed $s$ was found.

\begin{theo}[\cite{FKK}]\label{thm2} For every fixed $s$,
$\lim_{n\rightarrow \infty} \frac{\chi_{\ell} (K_n,s)}{\sqrt{sn}} = 1$.
\end{theo}

Since $\chi_{\ell} ({K_n})=\chi(K_n)=n$, Theorems~\ref{thm1} and~\ref{thm2} show that for
fixed $s$ and large $n$, $\chi_{\ell} (K_n,s)$ is much less than $\chi_{\ell} ({K_n})$.
In this paper, we study list colorings from $s$-separated lists of balanced complete multipartite graphs
and uniform hypergraphs. It turns out that even for small $s$, $\chi_{\ell} (K_{n,n},s)$
and $\chi_{\ell} (K_{n,n})$ are asymptotically the same.
 Let $K(k,m)=K_{m,m,...,m}$ denote the complete multipartite graph with $k$ partite sets  of size $m$.
One of our main results is

 \begin{theo}\label{t3'} For every fixed $k$, \\
 $ \chi_{\ell}(K(k,m),1)=(1+o(1))\chi_{\ell}(K(k,m))=(1+o(1))\log_{k/(k-1)} m$.
 \end{theo}

In view of (\ref{j1}), this means that for any $s$ and any fixed $k$,
$$\lim_{m\to\infty}\frac{\chi_{\ell}(K(k,m),s)}{\log_{k/(k-1)} m}=1.$$

We also prove a result of similar nature for balanced complete $k$-uniform $k$-partite hypergraphs.
Recall that a hypergraph is $k$-{\em partite} if its vertex set can be partitioned into $k$ sets
$V_1,\ldots,V_k$ so that each edge contains at most one vertex from each set.
 A $k$-partite (hyper)graph is {\em balanced} if all parts have equal sizes.

 Let $K^k(k,m)=K^k_{m,m,...,m}$ denote the complete $k$-uniform $k$-partite
hypergraph with  partite sets  of size $m$. Our second main result is:

 \begin{theo}\label{t4'} For every fixed $k$,\\
  $ \chi_{\ell}(K^k(k,m),1)=(1+o(1))\chi_{\ell}(K^k(k,m))=(1+o(1))\log_{k} m$ .
 \end{theo}

The upper bounds in Theorems~\ref{t3'} and~\ref{t4'} were known. To prove the lower bounds,
we need constructions of several uniform {\bf nearly disjoint} hypergraphs on the same vertex set
each of which has  small independence number. We show that such probabilistic constructions
 are possible and present them in the next section.
We think that these constructions are of interest by themselves.
Using these construction and the approach to list
 colorings used in~\cite{ERT} and later in~\cite{K}, we prove Theorems~\ref{t3'}
 and~~\ref{t4'} in Sections~\ref{s2} and~\ref{s3}, respectively.
Although we use only basic probabilistic tools, namely the first moment method, our asymptotics 
 for the nearly disjoint case are very close to the classical ones.

\section{Nearly disjoint hypergraphs with small independence number}\label{s1}

Hypergraphs $H_1$ and $H_2$ are \textit{nearly disjoint}
 if every edge of $H_1$ meets every edge of $H_2$ in at most one vertex.
Hypergraphs $H_1, H_2,\cdots, H_k$ are nearly disjoint if they are pairwise nearly disjoint.\medskip

 For a hypergraph $H$,  $\Delta(H)$ denotes the {\em maximum degree} of the vertices in $H$,
$\alpha(H)$ denotes the {\em independence number of} $H$, i.e. the size of  a largest subset of
 vertices of $H$ not containing edges of $H$, and $\tau(H)=|V(H)|-\alpha(H)$ denotes
  the {\em transversal number} of $H$.
  We are interested in constructing $k$ nearly disjoint hypergraphs each of size $m$ with
small independence numbers.  We first cite two results and prove a lemma
which we then use to construct these nearly disjoint hypergraphs.

  The following theorem is due to Erd\H os~\cite{E}. He proved it for the case when $k=2$,
but his proof can easily be extended to any fixed $k$.

\begin{theo}[\cite{E}]\label{thm3}
Let $k \ge 2$ be fixed. For $r$ sufficiently large, there exists $r$-uniform hypergraphs $H$
 on $n=\lceil \frac{k-1}{2}r^2 \rceil$ vertices with at most
 $\frac{e }{2}r^2k^{r-1}(k-1)\ln k$ edges such that $\alpha(H) < n/k$.
\end{theo}

  The following theorem is a partial case of a more general result by Lov{\'a}sz~(Corollary 2 in~\cite{L}).

\begin{theo}[\cite{L}]\label{thm4}
Let $u$ be the minimum size of an edge in a hypergraph $G$ with maximum degree $\Delta$. Then
$$\tau(G)\le (1+1/2+\cdots + 1/\Delta){|V(G)|}/u.$$
\end{theo}

For a positive integer $r$ and a real number $b$,  the binomial coefficient 
 $\binom{b}{r}$ is defined as $\frac{1}{r!}b(b-1)\cdots (b-r+1)$.\medskip

Let $a_1,\ldots,a_n$   be nonnegative integers and  $a= \max_{i} a_i$. 
If $a>0$, then
\begin{equation}\label{fl1}
\sum_{i=1}^{n}\binom{a_i}{2}\le \sum_{i=1}^{n}\frac{a_i(a-1)}{2}= \binom{a}{2}\frac{\sum_{i=1}^{n} a_i }{a} .
\end{equation}

We now prove our main lemma. 

\begin{lem}\label{lem1}
For each  $0< a < 1$ and integers $t\geq 1, r\geq 2$ and $q \geq  \frac{r^2}{a}  $,
there exists an $r$-uniform hypergraph
$H$ with $tq$ vertices  such that\smallskip

\noindent (i) $V(H) = T_1\cup T_2 \cup,...,\cup T_q$, where $|T_1| =\ldots =|T_q|= t$ and disjoint,\\
(ii) every edge in $H$ meets every $T_i$ in at most one vertex,\\
(iii)  $\alpha(H)< a t q$,\\
(iv) $|E(H)|\leq (1/a)^r\cdot 4tq$.
\end{lem}

\begin{proof}
Let $S = T_1\cup T_2 \cup,...,\cup T_q$. 
To prove the lemma, we will construct an auxiliary hypergraph $\mathcal{H}$ whose
vertex set is
$$V(\mathcal{H})= \bigcup_{1\le i_1<i_2\cdots <i_r\le q } T_{i_1}\times T_{i_2}\times \cdots \times T_{i_r}.$$
For every set $X\subset S$, we consider the set
  $E_X$ of $r$-sets that are contained in $X$ and members of $V(\mathcal{H})$.
The sets $E_X$ for every
  set $X$ with $|X|=\lceil a|S|\rceil$ will form the edges of $\mathcal{H}$.
A vertex cover of this
hypergraph $\mathcal{H}$ gives us a collection of $r$-subsets of $S$ with
the property that if we take any set $X\subset S$ with $|X|>a|S|$, then
we get an $r$-set which is entirely contained in $X$. A minimum vertex cover
of $\mathcal{H}$
gives us our required hypergraph $H$ with vertex set $V(H)=S$ and $|E(H)|=\tau(\mathcal{H})$.\\

We first estimate the size of $E_X$. Let $A_X$ denote the number of
$r$-subsets of $X$ that meet some $T_i$ in at least two vertices. (Note that we may assume that $t\ge 2$, since $A_X=0$ if $t=1$.)
Then
$$
A_X \le \sum_{i=1}^{q} \binom{|X\cap T_i|}{2} \binom{|X|-2}{r-2}
   \le \frac{|X|}{t}\binom{t}{2}\binom{|X|-2}{r-2}
   = \frac{1}{2}\binom{|X|}{r}\frac{(t-1)r(r-1)}{|X|-1}
   \le \frac{1}{2}\binom{|X|}{r},$$
where the second inequality is due to (\ref{fl1}) and the last inequality is by the choice of $q$, since
$$(t-1)r(r-1)\le (t-1)r^2\le a(t-1)q\le atq-1\le |X|-1.
  $$

\noindent Hence $$ |E_X| = \binom{|X|}{r}-A_X \ge \frac{1}{2}\binom{|X|}{r}. $$

Applying Theorem~\ref{thm4} for $G=\mathcal{H}$ and $u=\min |E_X|$ we get,
$$\tau({\mathcal{H}})\le \frac{|V(\mathcal{H})|}{\min |E_X|}(1+\ln \Delta)
 \le \frac{\binom{|S|}{r}}{\frac{1}{2}\binom{|X|}{r}}(1+ \ln 2^{|S|})
  \le \frac{\binom{|S|}{r}}{\frac{1}{2}\binom{a|S|}{r}}(|S|)$$
$$= 2 \Big(\frac{1}{a}\Big)^r\frac{(|S|)(|S|-1)\cdots(|S|-(r-1))}
                     {(|S|)(|S|-\frac{1}{a})\cdots(|S|-(\frac{r-1}{a}))}|S|\le
 2 \Big(\frac{1}{a}\Big)^r |S| \prod_{0\leq i\leq r-1} \frac{|S|-i}{|S|-\frac{1}{a}i}.$$
Here the last product is less than 2, since
$$ \prod_{0\leq i\leq r-1} \frac{|S|-i}{|S|-\frac{1}{a}i}
    = \prod_{0\leq i\leq r-1} \Big(1+ \frac{ (1-a)i}{a|S|-i}\Big)
 \le \exp \frac{(1-a)\sum_{i<r} i}{r^2-r+1}
  < \exp \frac{1-a}{2}.
  $$
Hence $$\tau(\mathcal{H})\le (1/a)^r\cdot 4tq.$$
\end{proof}

\medskip
\noindent \textbf{Construction}: Iterative Method for constructing nearly disjoint hypergraphs:\\

Let an integer $q \geq  \frac{r^2}{a} $ be fixed. We start with a $q$-vertex empty hypergraph.
We use Lemma~\ref{lem1} and obtain a hypergraph $G_1^1$ such that $\alpha(G_1^1)<a|V(G_1^1)|$.
After $i-1$ more iterations, we have hypergraphs $G_1^i, G_2^i,..., G_i^i$,
where $G_j^i$ is just $q$ vertex disjoint copies of $G_j^{i-1}$ (where $j<i$)
and $G_i^i$ is obtained by taking $q$ copies of $V(G_{i-1}^{i-1})$ and using Lemma~\ref{lem1}.\medskip

\noindent Note that we have the following:\medskip

\noindent   1. $G_\alpha^i, G_\beta^i$ are nearly disjoint for all $\alpha\neq \beta$;\\
  2. $|V(G_j^i)| = q^{i}$ for all $j\le i$;\\
  3. $|E(G_j^i)| \le (1/a)^r\cdot 4\cdot q^{i-1}q = (1/a)^r\cdot 4q^{i}$ for all $j\le i$;\\
  4. $\alpha(G_j^i)< a|V(G_j^i)|$ for all $j\le i$.\\

\noindent \textbf{Remark}: Note that it the above construction we took a $q$-vertex empty hypergraph and applied Lemma~\ref{lem1} to it to get $G_1^1$. One can instead (for $r$ sufficiently large) start with the hypergraph $G_1^1$ given in Theorem~\ref{thm3} and slightly improve the result.\medskip

\noindent We have the following corollary from the above construction.

\begin{cor}\label{cor1}
Let $k\ge 2, r\ge 2$, $0<a<1$ and $q = \lceil \frac{r^2}{a} \rceil$ .
There exist $k$ nearly disjoint $r$-uniform hypergraphs $H_1, H_2,\cdots H_k$ on the same vertex
set with $q^{k}$ vertices each with
$\lfloor 4q^{k}(\frac{1}{a})^r\rfloor$ edges such that $\alpha(H_i)< a|V(H_i)|$, for all $1\le i \le k$.
\end{cor}

\begin{proof}
From the construction metioned above we see that we have nearly disjoint hypergraphs $H_j=G_j^k$
such that  $\alpha(H_j)< a|V(H_j)|$ and  $|E(H_j)|\le 4q^{k}(\frac{1}{a})^r$,
for all $1\le j \le k$.
We just need to show we can add edges in $H_j$ such that
   $|E(H_j)|= \lfloor 4q^{k}(\frac{1}{a})^r\rfloor$.
It is in fact true that at every iteration step $i$, one can have $|E(G_j^i)| \leq 4q^{i}(\frac{1}{a})^r$, for all $1\le j \le i$, since at step $i$, we make $q$ copies of the vertex set with $q^{i-1}$ from the previous step and we have at least $\binom{q}{r}(q^{(i-1)})^{r}$ possibilites, which is much greater than $4q^{i}(\frac{1}{a})^r$.
\end{proof}

\noindent \textbf{Remark}: We required the sizes of the hypergraphs
 in the above corollary to be equal since the edges of these hypergraphs
  will form the list assignment for the vertices of a balanced multipartite graph.
But it is not necessary. 
We shall use Corollary~\ref{cor2} to generalize the result to unbalanced multipartite graphs.

\begin{cor}\label{cor2}
Let $k\ge 2, r\ge 2$, $0<a<1$ and $q = \lceil \frac{r^2}{a} \rceil$ .
There exist $k$ nearly disjoint $r$-uniform hypergraphs $H_1, H_2,\cdots H_k$ on the same vertex
set with $q^{k}$ vertices  with $|E(H_i)|$ obtaining any value in $[4q^{k}(\frac{1}{a})^r, 
 \binom{q}{r}q^{(i-1)r+(k-i)}]$ such that $\alpha(H_i)< a|V(H_i)|$, for all $1\le i \le k$.
\end{cor}

\begin{proof}
Consider the hypergraph $G_i^i$ obtained at Step $i$ in the construction. As we saw in Corollary~\ref{cor1}, for every $i$, $G_i^i$ can have at most $4 q^{i}(\frac{1}{a})^r$ edges. In fact, we can add all the possible $\binom{q}{r}(q^{(i-1)})^{r}$ edges. It is easy to see that we still maintain that the hypergraphs obtained so far in the construction are nearly disjoint. Moreover, we also do not increase the independence number of the hypergraphs by adding more edges. In the next $k-i$ steps of the construction we just take $q^{k-i}$ copies of $G_i^i$ to obtain $H_i$ which has at most $\binom{q}{r}q^{(i-1)r+(k-i)}$ edges.
\end{proof}

\section{Coloring complete multipartite graphs\\
 with $s$-separated lists}\label{s2}
Recall that
$K(k,m)=K_{m,m,...,m}$ denotes the complete multipartite graph with $k$ partite sets  of size $m$.
We will use the ideas of~\cite{ERT} and~\cite{K} and the results of the previous section
to prove Theorem~\ref{t3'}. For convenience, we restate it here.

 \begin{theo}\label{t3} For every fixed $k$, \\
 $ \chi_{\ell}(K(k,m),1)=(1+o(1))\chi_{\ell}(K(k,m))=(1+o(1))\log_{k/(k-1)} m$.
 \end{theo}

\begin{proof}
Let $G$ be a copy of $K(k,m)$ with partite sets $V_1,\ldots,V_k$.
 Let $L$ be an $r$-list for $G$.  Let $C:=\bigcup_{v\in V(G)}L(v)$. Then, since $G$ is complete $k$-partite,
  $G$ is $L$-colorable if and only if we
 can partition $C$ into sets $C_1,\ldots,C_k$ so that for each $i=1, \dots, k$ and each $v\in V_i$,
 $C_i\cap L(v)\neq \emptyset$. Let $H=H(G,L)$ be the $r$-uniform hypergraph with the vertex set $C$ whose edges
 are the lists of the vertices of $G$. Since lists of some vertices in $G$ may coincide, $H$ may have multiple edges.
For $i=1,\ldots,k$, let $E_i$ be the set of edges of $H$ that correspond to the lists of the vertices in $V_i$. So,
 $G$ is $L$-colorable if and only if

 (*) {\em we
 can color $V(H)$ with $k$ colors $1,\ldots,k$ so that for every $i$ and every edge $A\in E_i$, $A$ contains a vertex of
 color $i$.}

To get the upper bound we show the following statement.
\begin{equation}\label{q1}
\mbox{ Let $km<\left(\frac{k}{k-1}\right)^r$. Then $\chi_{\ell}(K(k,m),1)\leq \chi_{\ell}(K(k,m))\leq r$.}
\end{equation}

By the above, it is enough to prove that for $m<\left(\frac{k}{k-1}\right)^r$, every $r$-uniform hypergraph $H$
with $E(H)=E_1\cup\ldots\cup E_k$ where $|E_i|=m$ for $i=1,\ldots,m$ has a $k$-coloring  satisfying (*).
We color each $v\in V(H)$ randomly: $v$ gets color $i$ with probability $1/k$ independently from all other vertices.
An edge $A\in E_i$ is {\em happy} if some  vertex of $A$ gets color $i$, and {\em unhappy} otherwise.
For  each $A\in E(H)$, the probability that $A$ is unhappy is $(1-1/k)^r$. Thus the expectation
of the number of unhappy edges is at most
$km\left(\frac{k-1}{k}\right)^r<1$. So, there exists a coloring $f$ such that every edge is happy. This proves (\ref{q1}).

To prove the lower bound, observe that  $L$ is $1$-separated if and only if the
corresponding hypergraphs $H_1=(V(H),E_1),\ldots,H_k=(V(H),E_k)$
are nearly disjoint. Let $q:=\lceil \frac{kr^2}{(k-1)}\rceil$.
By Corollary~\ref{cor1} for $a=(k-1)/k$, there exist nearly disjoint $r$-uniform hypergraphs
$H_1,\ldots,H_k$ on the same vertex set, say $V$, such that for every $i=1, \dots, k$,\\
(a) $|E(H_i)|\leq 4(\frac{k}{k-1})^rq^k \leq 4(\frac{k}{k-1})^r 2^kr^{2k}$;\\
(b) $\alpha(H_i)<\frac{k-1}{k}|V|$.

We claim that the hypergraph $H:=\bigcup_{i=1}^kH_i$ does not satisfy (*). Indeed, suppose that
there is a
$k$-coloring $f$ such that for every $i$ and every edge $A\in E_i$,
$f^{-1}(i)\cap A\neq \emptyset$. We may assume that $|f^{-1}(1)|\geq\ldots\geq |f^{-1}(k)|$.
Let $B=V-f^{-1}(k)$. By our ordering, $|B|\geq \frac{k-1}{k}|V|$. So by (b), some edge
of $H_k$ is contained in $B$, a contradiction to the choice of $f$. Thus if
$k$ is fixed and positive integers $r$ and $m$ satisfy
$$m\geq 4(\frac{k}{k-1})^r(2r^2)^k,
$$
then $\chi_{\ell}(K(k,m),1)\geq 1+r$. Since for fixed $k$,
$$\ln \left(4(\frac{k}{k-1})^r(2r^2)^k\right)=
r\ln (\frac{k}{k-1})+ 2k\ln r+ (k+2)\ln 2=r\ln (\frac{k}{k-1})(1+o(1)),$$
the theorem is proved.
\end{proof}

We now state an easy generalization of Theorem~\ref{t3} for unbalanced multipartite graphs. It follows easily from Corollary~\ref{cor2} and the proof of Theorem~\ref{t3}.

\begin{theo}
Given positive integers $k, m_1, m_2, \cdots m_k$ such that $m_1\le m_2\le \cdots \le m_k$, let $r$ be the largest integer such that $m_1 \ge 4q^{k}(\frac{k}{k-1})^r$, where $q=\lceil \frac{k}{k-1}r^2 \rceil$.\smallskip

\noindent If $m_i \in [4q^{k}(\frac{k}{k-1})^r, \binom{q}{r}q^{(i-1)r+(k-i)}]$, for all $1\le i \le k$, then $\chi_{\ell}(K_{m_1, m_2,\cdots m_k},1) > r$.\smallskip

\noindent In other words $\chi_{\ell}(K_{m_1, m_2,\cdots m_k},1) \ge (1-o(1))\log_{k/(k-1)} m_1$
\end{theo}

\noindent \textbf{Remark}: One might want to show a matching upper bound or improve the lower bound and give a new bound in terms of $m_k$, when $m_k$ is not too large compared to $m_1$. A similar result was shown about the the choice number $\chi_{\ell}$ of unbalanced multipartite graphs in ~\cite{GK}.

\section{Coloring  complete $k$-uniform $k$-partite hypergraphs with $s$-separated lists}\label{s3}
Recall that
 $K^k(k,m)=K^k_{m,m,...,m}$ denotes the complete $k$-uniform $k$-partite
hypergraph with $k$ partite sets  of size $m$. In this section, we
prove Theorem~\ref{t4'}. For convenience, we restate it here.

 \begin{theo}\label{t4} For every fixed $k$,\\
  $ \chi_{\ell}(K^k(k,m),1)=(1+o(1))\chi_{\ell}(K^k(k,m))=(1+o(1))\log_{k} m$ .
 \end{theo}

\begin{proof}

Let $G$ be a copy of $K^k(k,m)$ with partite sets $V_1,\ldots,V_k$.
 Let $L$ be an $r$-list for $G$.   Let $C:=\bigcup_{v\in V(G)}L(v)$. Since $G$ is complete
 $k$-uniform $k$-partite, a coloring of $V(G)$ is proper if and only if no color is present on each $V_i$. Thus,
  $G$
   is $L$-colorable if and only    if we
 can partition $C$ into sets $C_1,\ldots,C_k$ so that for each $i=1, \dots, k$ and each $v\in V_i$,
 $C_i$ does not contain $L(v)$. As in the proof of Theorem~\ref{t3'}, let $H=H(G,L)$
 be the $r$-uniform hypergraph with the vertex set $C$ whose edges
 are the lists of the vertices of $G$.
For $i=1,\ldots,k$, let $E_i$ be the set of edges of $H$ that correspond to the lists of the vertices in $V_i$. So,
 $G$ is $L$-colorable if and only if

 (**) {\em we
 can color $V(H)$ with $k$ colors $1,\ldots,k$ so that for every $i$ and every edge $A\in E_i$, $A$ is
 not monochromatic of
 color $i$.}

First we prove:
\begin{equation}\label{q1'}
\mbox{ Let $m<k^{r-1}$. Then $\chi_{\ell}(K^k(k,m),1)\leq \chi_{\ell}(K^k(k,m))\leq r$.}
\end{equation}

It is enough to prove that for $m<k^{r-1}$, every $r$-uniform hypergraph $H$
with $E(H)=E_1\cup\ldots\cup E_k$ where $|E_i|=m$ for $i=1,\ldots,m$ has a $k$-coloring  satisfying (**).
We color each $v\in V(H)$ randomly: $v$ gets color $i$ with probability $1/k$ independently from all other vertices.
An edge $A\in E_i$ is {\em happy} if some  vertex of $A$ gets color distinct from $i$, and {\em unhappy} otherwise.
For  each $A\in E(H)$, the probability that $A$ is unhappy is $k^{-r}$. Thus the expectation
of the number of unhappy edges is at most
$kmk^{-r}<1$. So, there exists a coloring $c$ such that every edge is happy. This proves (\ref{q1'}).

Now we prove the lower bound. Recall that
  $L$ is $1$-separated if and only if the
corresponding hypergraphs $H_1=(V(H),E_1),\ldots,H_k=(V(H),E_k)$
are nearly disjoint. Let $q:=kr^2$.
By Corollary~\ref{cor1} for $a=1/k$, there exist nearly disjoint $r$-uniform hypergraphs
$H_1,\ldots,H_k$ on the same vertex set, say $V$ such that for every $i=1,\dots, k$,\\
(a) $|E(H_i)|\leq 4k^rq^k$;\\
(b) $\alpha(H_i)<\frac{1}{k}|V|$.

We claim that the hypergraph $H:=\bigcup_{i=1}^kH_i$ does not satisfy (**). Indeed, suppose that
there is a
$k$-coloring $f$ such that for every $i$ and every edge $A\in E_i$,
$A\not\subseteq f^{-1}(i)$. We may assume that $|f^{-1}(1)|\geq |V|/k$.
 Then  by (b), some edge
of $H_1$ is contained in $f^{-1}(1)$, a contradiction to the choice of $f$. Thus if
$k$ is fixed and positive integers $r$ and $m$ satisfy
$$m\geq 4k^r(kr^2)^k,
$$
then $\chi_{\ell}(K^k(k,m),1)\geq 1+r$. Since for fixed $k$,
$$\ln \left(4k^r(kr^2)^k\right)=
(r+k)\ln k+2k\ln r+2\ln 2=r\ln k(1+o(1)),$$
the theorem is proved.
\end{proof}

\end{document}